\documentclass[11pt]{article}
\usepackage{graphicx,amsmath,amssymb,amsfonts,authblk}
\date{}
\newtheorem{theorem}{Theorem}[section]

\newtheorem{lemma}[theorem]{Lemma}

\newenvironment{proof}[1][Proof]{\noindent\textbf{#1.} }{\ \rule{0.5em}{0.5em}}

\begin{document}

\title{A case of the generalized vanishing conjecture}
\author{Zhenzhen Feng\footnote{Email: fengzz13@mails.jlu.edu.cn}, Xiaosong Sun\footnote{Corresponding author, Email: sunxs@jlu.edu.cn}}
\affil{\small School of Mathematics, Jilin University, Changchun 130012, China}

\maketitle

\begin{abstract}
In this paper, we show that the GVC (generalized vanishing conjecture) holds for the differential operator $\Lambda=(\partial_x-\Phi(\partial_y))\partial_y$ and all polynomials $P(x,y)$, where $\Phi(t)$ is any polynomial over the base field. The GVC arose from the study of the Jacobian conjecture.
\\{\bf{\emph{Key words}}}:  Jacobian conjecture, generalized vanishing conjecture, differential operators
\\{\bf{\emph{MSC}}}: 14R15, 13N15
\end{abstract}

\section{Introduction}

The well-known Jacobian conjecture (JC for short) was first proposed by O. Keller in 1939 (see \cite{BSD} and \cite{E}), which asserts that any polynomial map $F$ from the complex affine $n$-space $\mathbb{C}^n$ to
itself with $\det JF=1$ must be an automorphism of $\mathbb{C}^n$. Various special cases of this still
mysterious conjecture have been investigated, and connections with some other notable problems have been established. For example,
the JC is related to some problems in combinatorics (cf. \cite{W}) and the JC is equivalent to the Dixmer conjecture proposed by Dixmier \cite{D} (cf. \cite{T, BK, AE}) and also to the Mathieu conjecture proposed by Mathieu \cite{M} in 1995.

 It was shown independently by de Bondt and van den Essen \cite{BE} and Meng \cite{G} that for the JC one only need to consider all polynomial maps of the form $X+H: \mathbb{C}^n\rightarrow \mathbb{C}^n$ for all dimensions $n$, where $H$ is cubic homogeneous and $JH$ is symmetric and nilpotent. Based on this result, Zhao proposed in 2007 the vanishing conjecture (see \cite{Z2, EZ}) and generalized it later in \cite{Z1} to the following form.
\vskip 0.3cm

{\it \textbf{\textup{Generalized Vanishing Conjecture (GVC(n))}} Let $\Lambda$ be any differential operator on $\mathbb{C}[z]:=\mathbb{C}[z_1,z_2,\ldots,z_n]$ with constant coefficients. If $P\in \mathbb{C}[z]$ is  such that $\Lambda^m(P^m)=0$ for all $m\geq 1$, then for any polynomial $Q\in \mathbb{C}[z]$, we have $\Lambda^m(P^mQ)=0$ for all $m\gg 0$.}
\vskip 0.3cm

In fact, Zhao showed in \cite{Z2, EZ} that the JC holds for all dimensions $n$ if and only if the GVC holds for all dimensions $n$ for the case where $\Lambda$ is the Laplace operator $\sum_{i=1}^n\partial_{z_i}^2$ and $P$ is homogeneous.

Up to now, the GVC($n$) was verified in the following special cases:  (1) $n=1$; (2) $n=2$ and $\Lambda=\partial_{z_1}-\Phi(\partial_{z_2})$; (3) $\Lambda(t)$ (or $P(z)$) is a linear combination of two monomials
with different degrees (see \cite{EWZ}); (4) $n\leq 4$, $\Lambda$ is the Laplace operator; (5) $n=5$, $\Lambda$ is the Laplace operator and $P$ is homogeneous (due to \cite{BE04, BE05} and \cite{Z1}).

In this paper, we showed that the GVC holds for the differential
operator $\Lambda=(\partial_x-\Phi(\partial_y))\partial_y$ on $\mathbb{C}[x,y]$ and all polynomials
$P(x,y)\in \mathbb{C}[x,y],$ where $\Phi(t)$ is any polynomial over $\mathbb{C}$. The conclusion is in fact valid for any field of characteristic zero.

\section{The proof of GVC for $\Lambda=(\partial_x-\Phi(\partial_y))\partial_y$}
Throughout this section, $K$ stands for a field of characteristic zero.
For simplicity, we write $K[x,y]$ instead of $K[z_1,z_2]$.  We consider the GVC for the differential operator $\Lambda=(\partial_x-\Phi(\partial_y))\partial_y$, where $\Phi(t)$ is an arbitrary polynomial over $K$. We write $\Phi(t)$ as
$$\Phi(t)=q_{0}+q_{1}t+\cdots+q_{s}t^s$$
where $q_{i}\in K,0\leq i\leq s$. And we denote by $o(\Phi(t))$ or $o(\Phi)$ the order of the polynomial $\Phi(t)$, i.e. the
least integer $m\geq 0$ such that $q_{m}\neq 0$.

We will show the following theorem.
\begin{theorem} \label{th1}
The GVC holds for the differential operator $\Lambda=(\partial_{x}-\Phi(\partial_{y}))\partial_{y}$
and all polynomials $P(x,y)\in K[x,y]$.
\end{theorem}

We start with some lemmas.

\begin{lemma}
Let $\Lambda$ be as above and let $0\neq P(x,y)\in K[x,y]$ be such that $\Lambda P=0$.  Then
\vskip 0.2cm

(1) $q_{0}=0$ (i.e. $\Phi(t)= 0$ or $o(\Phi)> 1$)
or $P(x,y)\in K[x]$;
\vskip 0.2cm

(2) $P(x,y)=e^{x\Phi(\partial_{y})}(f(x)+g(y))$ for some $f(x)\in K[x]$ and $g(y)\in K[y]$.
\end{lemma}

\begin{proof}
(1) Suppose that $q_0\neq 0$. Note that
$$\Lambda=(\partial_{x}-\Phi(\partial_{y}))\partial_{y}=\partial_{x}\partial_{y}
-(q_{0}+q_{1}\partial_{y}+\cdots+q_{s}\partial_{y}^s)\partial_{y}.$$
Since $\Lambda P=0$, the highest  homogeneous part of $\Lambda P$ is zero, i.e.,
$q_{0}\partial_{y}P=0$, and thus $q_{0}=0$ or $\partial_{y}P=0$ i.e. $P(x,y)\in K[x]$.

(2) Firstly, observe that
\begin{eqnarray*}
         \partial_{x}\partial_{y}(e^{-x\Phi(\partial_{y})}P)
         &=&e^{-x\Phi(\partial_{y})}(\partial_{x}-\Phi(\partial_{y}))\partial_{y}P\\
         &=&e^{-x\Phi(\partial_{y})}(\Lambda P)\\&=&0.
\end{eqnarray*}
So there are no terms $x^ay^b$ with $a\geq1$ and $b\geq 1$ in $e^{-x\Phi(\partial_{y})}P$, namely
$$e^{-x\Phi(\partial_{y})}P=f(x)+g(y)$$
for some $f(x)\in K[x]$ and $g(y)\in K[y]$.
Applying $e^{x\Phi(\partial_{y})}$ to both sides of the last equation, we obtain that
$P(x,y)=e^{x\Phi(\partial_{y})}(f(x)+g(y)).$
\end{proof}

Now we write $f(x)$ and $g(x)$ above as
$$f(x)=a_{0}+a_{1}x+a_{2}x^{2}+\cdots+a_{s}x^{s},$$
$$g(y)=b_{0}+b_{1}y+b_{2}y^{2}+\cdots+b_{d}y^{d},$$
where $s\geq0,a_{j},b_{t}\in K ,0\leq j\leq s,0\leq t\leq d$. We may assume that $a_{0}=0$.

\begin{lemma}
Let $\Lambda$ and $0\neq P(x,y)\in K[x,y]$ be as above with $\Lambda P=0$. If
$r:=o(\Phi)\geq 2$ and $\Lambda P=\Lambda^{2}(P^{2})=0$, then $o(\Phi)\geq \deg g$ and $P(x,y)=a_1x+g(y)$ for some $a_1\in K.$
\end{lemma}

\begin{proof}
Note that
\begin{align*}
\nonumber P^{2}&=\big(e^{x\Phi(\partial_{y})}(f(x)+g(y))\big)^{2}\\&=f^{2}+g^{2}+x^{2}(\Phi(g))^{2}+\frac{1}{4}x^{4}({\Phi}^{2}(g))^{2}+2fg+2xf\Phi(g)+x^{2}f{\Phi^2}(g)\\
\nonumber&+2xg\Phi(g)+x^{2}g{\Phi}^{2}(g)+x^{3}\Phi(g){\Phi}^{2}(g)+\frac{1}{3}x^{3}(f+g)\Phi^{3}(g)+\cdots.
\end{align*}
Viewing $\Lambda^{2}(P^{2})=(\partial_{x}-\Phi(\partial_{y}))^{2}{\partial_{y}^{2}}(P^2)$
as a polynomial in $K[y][x]$, and looking at its constant term, we obtain that
\begin{eqnarray}
\nonumber0=\Lambda^{2}P^{2}|_{x=0}&=&({\partial_{x}}^{2}-2{\partial_{x}}\Phi +\Phi^{2})[(g^2)''+x^{2}{(\Phi(g))^{2}}''
+\frac{1}{4}x^{4}{(\Phi^{2}(g))^{2}}''\\
\nonumber&&+2fg''+2xf{\Phi(g)}''+x^{2}f{(\Phi^{2}(g))}''+2x{(g\Phi(g))}''\\
\nonumber&&+x^{2}{(g\Phi^{2}(g))}''+x^{3}{(\Phi(g)\Phi^{2}(g))}''+\frac{1}{3}x^{3}(f+g)\Phi^{3}(g)'']\\
\nonumber&=&\Phi^{2}{(g^{2})''}-4a_{1}\Phi(g''))\\
\nonumber&&-4\Phi((g\Phi(g))'')
+2{(\Phi(g)^{2})''}+4a_{2}g''\\
&&+4a_{1}{(\Phi(g))''}+2{(g\Phi^{2}(g))''}. \label{eq1}
\end{eqnarray}

By the hypothesis of the lemma, $r:=o(\Phi(t))\geq 2$.

{\textbf{Claim:}} $r\geq d:=\deg g$.

If $d>2r$, then the highest degree on $y$ in (\ref{eq1}) is $2d-2r-2$ with
coefficient
\begin{align*}{q_{r}}^{2}{b_{d}}^{2}\Big(\frac{(2d)!}{(2d-2r-2)!}&-4\frac{d!(2d-r)!}{(d-r)!(2d-2r-2)!}
+2\frac{d!d!(2d-2r)(2d-2r-1)}{(d-r)!(d-r)!}\\&+\frac{d!(2d-2r)(2d-2r-1)}{(d-2r)!}\Big).
\end{align*}
It is easy to see that, in the last formula, the first term is greater than the second one,
and thus the coefficients is non-zero, a contradiction. Thus $d\leq 2r$.

Observe that when $d<2r$, we have $d-2>2d-2r-2$, whence the term in (\ref{eq1}) with the highest degree on $y$ is $4a_{2}g''$,
so we must have $a_{2}=0$.

If $2r>d>r$, then $2d-2r-2>d-r-2$, whence the highest degree on $y$ is $2d-2r-2$, which is
impossible as discussed above.

Therefore, $d\leq r$ or $d=2r$.

Now we show that $d\neq 2r$. Suppose conversely that $d=2r$. Then the coefficient of the term $x$ in
\begin{eqnarray}
\nonumber
 \Lambda^{2}P^{2}&=&({\partial_{x}}^{2}-2\partial_{x}\Phi+\Phi^{2}){\partial_{y}}^{2}
[(1+x\Phi+\frac{1}{2}x^{2}\Phi^{2}+\frac{1}{3!}x^{3}\Phi^{3}+\cdots)(f(x)+g(y))]^{2}\\
\nonumber&=&({\partial_{x}}^{2}-2\partial_{x}\Phi+\Phi^{2})
[(g^{2})''+x^{2}{(\Phi(g))^{2}}''+2fg''+2xf{(\Phi(g))''}\\
\nonumber&&+x^{2}f{(\Phi^{2}(g))}''
+\frac{1}{3}x^{3}f(\Phi^{3}(g))''+x^{3}(\Phi(g)\Phi^{2}(g))''],
\end{eqnarray}
is {\small
\begin{align*}h(y)&:=2a_{1}\Phi^{2}(g'')+2\Phi^{2}[(g\Phi(g))'']-4\Phi[{(\Phi(g))^{2}}'']
-8a_{2}\Phi(g'')-8a_{1}\Phi[(\Phi(g))'']\\&\small -4\Phi[(g\Phi^{2}(g))'']
+12a_{3}g''+12a_{2}(\Phi(g))''+6a_{1}(\Phi^{2}(g))''
+2(g\Phi^{3}(g))''+6(\Phi(g)\Phi^{2}(g))''.
\end{align*}}
Since $d=2r$, the term with the highest degree on $y$ in $h(y)$ is contained in
{\small$$p(y)=2\Phi^{2}[(g\Phi(g))'']-4\Phi[{(\Phi(g))^{2}}'']-8a_{2}\Phi(g'')
-4\Phi[(g\Phi^{2}(g))'']+12a_{2}(\Phi(g))''+6(\Phi(g)\Phi^{2}(g))'',$$} and its coefficient is
\begin{eqnarray*}
  s&=&b_{d}^{2}q_{r}^{3}(2\frac{d!(2d-2r)!}{(d-r)!(2d-3r-2)!}-4\frac{d!d!(2d-2r)!}{(d-r)!(d-r)!(2d-3r-2)!}\\
  \nonumber&&-4\frac{d!(2d-2r)!}{(d-2r)!(2d-3r-2)!}+6\frac{d!d!(2d-3r)(2d-3r-1)}{(d-2r)!(d-r)!})\\
\nonumber&&-8a_{2}b_{d}q_{r}\frac{d!}{(d-r-2)!}+12a_{2}b_{d}q_{r}\frac{d!(d-r)(d-r-1)}{(d-r)!}\\
\nonumber&=&b_{d}q_{r}\frac{(2r)!}{(r-2)!}[b_{d}q_{r}^{2}(\frac{2(3r)!}{r!}-4\frac{(2r)!(2r)!}{r!r!}+2(2r)!)+4a_{2}]\\
\nonumber&=&0.
\end{eqnarray*}
And the constant of $h(y)$ is
$$t=\frac{b_{d}}{(2r)(2r-1)}[b_{d}q_{r}^{2}(\frac{(4r)!}{(2r)!}-\frac{4(3r)!}{r!}+\frac{2(2r)!(2r)!}{r!r!}
+2(2r)!)+4a_{2}]
=0.$$
Since $s=t=0$, we get that
$$\frac{2(3r)!}{r!}-\frac{4(2r)!(2r)!}{r!r!}+2(2r)!
=\frac{(4r)!}{(2r)!}-\frac{4(3r)!}{r!}+\frac{2(2r)!(2r)!}{r!r!}+2(2r)!,$$
or equivalently,
\begin{align} \label{eq2}
(4r)!r!r!-6(3r)!(2r)!r!+6(2r)!(2r)!(2r)!=0,
\end{align}
Then $(4r)!r!r!-6(3r)!(2r)!r!<0$, which implies that $r=2$. Then the equation (\ref{eq2}) becomes $4!\cdot4!\cdot(64)=0$, a contradiction, so $d\neq 2r$.

Thus we have prove the claim that $d\leq r$.

In the case $d<r$, we have $P(x,y)=f(x)+g(y)$
for some $f(x)\in K[x]$ and $g(y)\in K[y]$.

In the case $d=r$, we have $P(x,y)=f(x)+g(y)+x\Phi(g)$. One may observe that $x\Phi(g)\in K[x]$, and thus in this case $P(x,y)$ is also of the form $f(x)+g(y)$.

Observing that
\begin{eqnarray*}
0=\Lambda^{2}P^{2}&=&(\partial_{x}-\Phi(\partial_{y}))^{2}\partial_{y}^{2}(f^{2}+2fg+g^{2})\\
                  &=&(\partial_{x}^{2}-2\partial_{x}\Phi(\partial_{y})+\Phi^{2}(\partial_{y}))(2fg''+(g^{2})'')\\
                  &=&2g''\partial_{x}^{2}(f),
 \end{eqnarray*}
we have $f(x)=a_1x$, $a_1\in K.$
\end{proof}

Now we are in the position to  prove Theorem \ref{th1}.\\

{\bf Proof of Theorem \ref{th1}:} The case $P=0$ is obvious and thus we  assume that $P\neq 0$. And we may assume that $o(\Phi)\geq 2$ through coordinate change. Suppose that $\Lambda^{m}(P^{m})=0$, $\forall m
\geq 1$. We need to show that, for any $h\in K[z]$, we have $\Lambda^{m}(P^{m}h)=0$, $\forall m\gg 0$. It is suffices to take $h(x,y)=x^{a}y^{b},a\geq 0,b\geq 0$, whence
\begin{eqnarray*}
\Lambda^{m}(P^{m}h)&=&(\partial_{x}-\Phi(\partial_{y}))^{m}\partial_{y}^{m}((f+g)^{m}x^{a}y^{b})\\
                        &=&\Big(\sum_{i=0}^{m}(-1)^{i}C_{m}^{i}\partial_{x}^{m-i}\Phi^{i}(\partial_{y})\Big)
                        \partial_{y}^{m}\big(\sum_{j=0}^{m}C_{m}^{j}f^{m-j}g^{j}x^{a}y^{b}\big)\\
&=&\sum_{i=0}^{m}\sum_{j=0}^{m} (-1)^{i}C_{m}^{i}\partial_{x}^{m-i}\Phi^{i}(\partial_{y})\cdot \partial_{y}^{m}C_{m}^{j}f^{m-j}g^{j}x^{a}y^{b}.
 \end{eqnarray*}
When $m-i>m-j+a$, we have $$(-1)^{i}C_{m}^{i}\partial_{x}^{m-i}\Phi^{i}(\partial_{y})\cdot \partial_{y}^{m}C_{m}^{j}f^{m-j}g^{j}x^{a}y^{b}=0.$$
When $m-i\leq m-j+a$, i.e., $a+i\geq j$, we have
$$o(\Phi^{i}(\partial_{y}))=ri\mbox{~~and~~}
\deg_{y}(\partial_{y}^{m}f^{m-j}g^{j}x^{a}y^{b})=dj+b-m,$$
If $m>b+ar$, then $$m>b+ar\geq b+(j-i)r\geq b+dj-ir$$ and thus
$$\Lambda^{m}(P^{m}h)=0,\ \textrm{ }\forall m>b+ar,$$
which completes the proof.

\paragraph{Acknowledgments} The paper is a part of the first author's Master's Thesis under the supervision of the second author.

\end{document}